%% file: main.tex
\documentclass[a4paper]{article}

\input{settings}

\author{Mark C. Bell\\ University of Illinois\\ \texttt{mcbell@illinois.edu}}

\title{Asymmetric dynamics of outer automorphisms}

\begin{document}

\maketitle

\begin{abstract}
We consider the action of an irreducible outer automorphism $\phi$ on the closure of Culler--Vogtmann Outer space.
This action has north-south dynamics and so, under iteration, points converge exponentially to $[T^\phi_+]$.

For each $N \geq 3$, we give a family of outer automorphisms $\phi_k \in \Out(\FF_N)$ such that as, $k$ goes to infinity, the rate of convergence of $\phi_k$ goes to infinity while the rate of convergence of $\phi_k^{-1}$ goes to one.
Even if we only require the rate of convergence of $\phi_k$ to remain bounded away from one, no such family can be constructed when $N < 3$.

This family also provides an explicit example of a property described by Handel and Mosher: that there is no uniform upper bound on the distance between the axes of an automorphism and its inverse.
\end{abstract}

\keywords{Outer automorphism, rate of convergence, asymmetric.}

\ccode{37E25, 20E05, 37D20}

\bigskip \bigskip

An irreducible outer automorphism $\phi \in \Out(\FF_N)$ acts on $\overline{\outerspace_N}$, the closure of Culler--Vogtmann Outer space \cite[Section~1.4]{Vogtmann}, with north-south dynamics \cite[Theorem~1.1]{LevittLustig}.
Therefore its action has a pair of fixed points $[T^\phi_+], [T^\phi_-] \in \partial \outerspace_N$ and under iteration points in $\overline{\outerspace_N}$ (other than $[T^\phi_-]$) converge to $[T^\phi_+]$.

There is a natural embedding of $\overline{\outerspace_N}$ into $\RP^\calC$, where $\calC$ is the set of conjugacy classes of the free group $\FF_N$.
In this embedding, the $w \in \calC$ coordinate of a marked graph $(G, g) \in \outerspace_N$ is given by the length of the shortest loop in $G$ which is freely homotopic to $g(w)$ \cite[Page~5]{Vogtmann}.
In this coordinate system the convergence to $[T^\phi_+]$ is exponential.
To measure the rate of this convergence we recall two definitions from \cite{BellSchleimer}.

\begin{definition}[{\cite[Definition~1.2]{BellSchleimer}}]
Suppose that $f \in \ZZ[x]$ is a polynomial with roots $\lambda_1, \ldots, \lambda_m$, ordered such that $|\lambda_1| \geq |\lambda_2| \geq \cdots \geq |\lambda_m|$. The \emph{spectral ratio} of $f$ is $\spectralratio(f) \defeq |\lambda_1 / \lambda_2|$.
\end{definition}

\begin{definition}
Suppose that $\phi \in \Out(\FF_N)$ is an irreducible outer automorphism.
Let $\mu_{\stretchfactor(\phi)} \in \ZZ[x]$ denote the minimal polynomial of its stretch factor $\stretchfactor(\phi)$.
The \emph{spectral ratio} of $\phi$ is $\spectralratio(\phi) \defeq \spectralratio(\mu_{\stretchfactor(\phi)})$.
\end{definition}

As described in \cite{BellSchleimer}, the rate of convergence of $[T \cdot \phi^n]$ to $[T^\phi_+]$ is determined by $\spectralratio(\phi)$.
Thus we state the main result of this paper, a complete characterisation of when it is possible to build outer automorphisms which converge rapidly in one direct but slowly in the other, in terms of the spectral ratio:

\begin{theorem}
\label{thrm:main}
There is a family of fully irreducible outer automorphisms $\phi_k \in \Out(\FF_N)$ such that $\spectralratio(\phi_k) \to \infty$ but $\spectralratio(\phi_k^{-1}) \to 1$ if and only if $N \geq 3$.
\end{theorem}

This gives another difference between irreducible outer automorphisms and pseudo-Anosov mapping classes of surfaces.
A pseudo-Anosov mapping class $h \in \Mod(S)$ has its spectral ratio $\spectralratio(h)$ defined in terms of its dilatation \cite[Definition~1.3]{BellSchleimer}.
However, $h$ and $h^{-1}$ have the same dilatation \cite[Proposition~11.3]{FarbMargalit} and so $\spectralratio(h) = \spectralratio(h^{-1})$ automatically.

It is straightforward to prove the forward direction of Theorem~\ref{thrm:main} by considering its contrapositive:
\begin{itemize}
\item When $N = 1$ there is nothing to check as all automorphisms of $\FF_1$ are finite order.
\item When $N = 2$, any irreducible outer automorphism $\phi \in \Out(\FF_2)$ is geometric.
Hence there is a pseudo-Anosov mapping class on $S_{1,1}$, the once-punctured torus, which induces $\phi$ on $\pi_1(S_{1,1})$.

However, the spectral ratio of any pseudo-Anosov mapping class of the once-punctured torus is at least $\varphi^4 = 6.854101\cdots$, where $\varphi$ is the golden ratio \cite[Section~3.3]{BellSchleimer}.
Therefore $\spectralratio(\phi) = \spectralratio(\phi^{-1}) \geq \varphi^4$ and so these are both bounded away from one.
\end{itemize}
Thus we devote the remainder of this paper to constructing an explicit family of outer automorphisms of $\FF_N = \langle a_1, a_2, \ldots, a_N \rangle$, for fixed $N \geq 3$.

To do this, we start by considering the polynomials
\[ f(x,y) \defeq x^N - yx^{N-1} - 1 \inlineand g(x, y) \defeq x^N - yx^{N-2} - 1 \]
in $\QQ[x, y]$.
Since these are linear polynomials in $y$, they are irreducible over $\QQ$.
Therefore, by Hilbert's irreducibility theorem \cite[Chapter~9]{Lang}, there are infinitely many integers $k \geq 3$ such that $p_k(x) \defeq f(x, k)$ and $q_k(x) \defeq g(x, k)$ are both irreducible.

\begin{lemma}
\label{lem:root_pos}
Suppose that $k \geq 3$.
The polynomial $p_k$ has $N-1$ roots inside of the unit circle and one root in $[k, k+1]$.
Similarly, the polynomial $q_k$ has $N-2$ roots inside of the unit circle, one root in $[-\sqrt{k+1}, -\sqrt{k-1}]$ and one root in $[\sqrt{k-1}, \sqrt{k+1}]$.
For example, see Figure~\ref{fig:polynomials}.
\end{lemma}

\begin{proof}
Let $h(z) \defeq -k z^{N-1}$.
Then when $|z| = 1$ we have that
\[ |(p_k-h)(z)| = |z^N - 1| \leq 2 < |p_k(z)| + |h(z)|. \]
Therefore, by Rouch\'{e}'s theorem, $p_k$ must have the same number of roots inside of the unit circle as $h$ does.
Hence $p_k$ has $N-1$ roots inside of the unit circle.

Furthermore, since $\deg(p_k) = N$, there is only one more root to locate.
The intermediate value theorem now shows it must lie in $[k, k+1]$.

Applying the same argument with $h(z) \defeq -k z^{N-2}$ to $q_k$ shows it must have $N-2$ roots inside of the unit circle.
Again, we verify that the other two roots of $q_k$ lie in $[-\sqrt{k+1}, -\sqrt{k-1}]$ and $[\sqrt{k-1}, \sqrt{k+1}]$ by using the intermediate value theorem.
\end{proof}

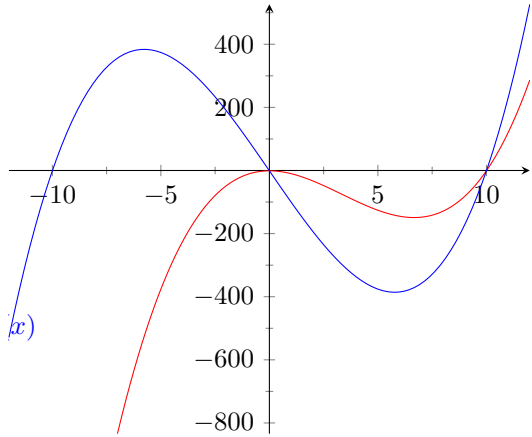
\begin{figure}[ht]
\centering
\begin{tikzpicture}
  \begin{axis}[minor tick num=1, axis lines=middle, label style={at={(ticklabel cs:1.1)}}]
    \addplot[red, domain=-7:12, samples=100]{x^3 - 10*x^2 - 1} node[xshift=-15pt, pos=0.25] {$f_{10}(x)$} ;
    \addplot[blue, domain=-12:12, samples=100]{x^3 - 100*x - 1} node[xshift=-5pt, yshift=40pt, pos=0.75] {$g_{100}(x)$} ;
  \end{axis}
\end{tikzpicture}
\caption{The polynomials $p_{10}$ and $q_{100}$ when $N = 3$.}
\label{fig:polynomials}
\end{figure}

In fact knowing the positions of the roots allows us to show many of these polynomials are irreducible directly.

\begin{lemma}
Whenever $k \geq 3$, the polynomial $p_k$ is irreducible over $\ZZ$.
\end{lemma}

\begin{proof}
First assume that there is a $z$ such that $|z| = 1$ and $p_k(z) = 0$ and so $1 = |z^N - k z^{N-1}| = |z - k|$.
However this would mean that $|k| \leq |k - z| + |z| = 2$ which contradicts the fact that $k \geq 3$.
Hence $p_k$ cannot have a root on the unit circle.

Now assume that $p_k$ is reducible.
Then one of its factors, $q \in \ZZ[x]$, must have all of its roots inside of the unit circle.
Therefore the constant term of $q$, which is the product of its roots, must have modulus less than one.
However this means that the constant term of $q$ must be zero and so zero is a root of $p_k$, which is false.
\end{proof}

Similarly, by taking into account the symmetry of the roots, the same argument also shows that $q_k$ is irreducible whenever $k \geq 3$ and $N$ is even.
Furthermore, for many of the low, odd values of $N$ we can find a prime $p$ and $k' \in \ZZ / p\ZZ$ such that the image of $q_{k'}$ in $(\ZZ / p \ZZ)[x]$ is irreducible.
It then follows that $q_k$ is irreducible whenever $k \equiv k' \pmod{p}$.
Some of these values are shown in Table~\ref{tab:irreducible}.

\begin{table}[ht]
\centering
\begin{tabular}{c|cccccccccccccc}
$N$ & 3 & 5 & 7 & 9 & 11 & 13 & 15 & 17 & 19 & 21 & 23 & 25 & 27 & 29 \\ \hline
$p$ & 2 & 2 & 3 & 3 & 2 & 7 & 3 & 7 & 3 & 2 & 5 & 5 & 3 & 2 \\
$k'$ & 1 & 1 & 1 & 0 & 1 & 4 & 1 & 4 & 1 & 1 & 4 & 0 & 0 & 1
\end{tabular}
\caption{$q_k$ is irreducible whenever $k \equiv k' \pmod{p}$.}
\label{tab:irreducible}
\end{table}

Now, following \cite[Page~3154]{HandelMosher}, let $\phi_k \in \Out(\FF_N)$ be the outer automorphism given by:
\[ a_N \mapsto a_{N-1} \mapsto a_{N-2} \mapsto \cdots \mapsto a_2 \mapsto a_1 \mapsto a_1^k a_N. \]
We use this family to conclude the remaining direction of Theorem~\ref{thrm:main}:

\begin{theorem}
\label{thrm:spectral_ratio_higher}
Suppose that $k \geq 3$ is such that $p_k$ and $q_k$ are both irreducible.
The outer automorphisms $\phi_k$ and $\phi_k^{-1}$ are non-geometric and fully irreducible.
Furthermore, $\spectralratio(\phi_k) \geq k$ while $\spectralratio(\phi_k^{-1}) < 1 + 1 / \sqrt{k}$.
\end{theorem}

\begin{proof}
We interpret $\phi_k$ as a piecewise-linear homotopy equivalence $f$ of the $N$--petalled rose.
Since this is a positive automorphism, this is a train track map \cite{BestvinaHandel}.
The transition matrix of this map is:
\[ A \defeq \left(\begin{array}{cccccc}
k & 0 & 0 & \hdots & 0 & 1 \\
1 & 0 & 0 & \hdots & 0 & 0 \\
0 & 1 & 0 & \hdots & 0 & 0 \\
0 & 0 & 1 & \hdots & 0 & 0 \\
\vdots & \vdots & \vdots & \ddots & \vdots & \vdots \\
0 & 0 & 0 & \hdots & 1 & 0
\end{array} \right) \]
and, up to a sign, the characteristic polynomial of $A$ is
\[ x^N - k x^{N-1} - 1 = p_k(x). \]
We chose $k$ so that this polynomial is irreducible and so $\spectralratio(\phi_k) = \spectralratio(p_k)$.
Now by Lemma~\ref{lem:root_pos} this polynomial has $N-1$ roots that lie in the interior of the unit circle and one root in $[k, k+1]$.
Hence $\stretchfactor(\phi_k) \approx k$ and $\spectralratio(\phi_k) = \spectralratio(p_k) \geq k$.

On the other hand, $\phi_k^{-1}$ is given by:
\[ a_1 \mapsto a_2 \mapsto a_3 \mapsto \cdots \mapsto a_{N-1} \mapsto a_N \mapsto a_{N-1}^{-k} a_1. \]
Again, we consider this as a piecewise-linear homotopy equivalence $g$ of the $N$--petalled rose.
Direct calculation of the turns involved shows that this is again a train track map and that its transition matrix is:
\[ B \defeq \left(\begin{array}{cccccc}
0 & 1 & 0 & \hdots & 0 & 0 \\
0 & 0 & 1 & \hdots & 0 & 0 \\
0 & 0 & 0 & \hdots & 0 & 0 \\
\vdots & \vdots & \vdots & \ddots & \vdots & \vdots \\
0 & 0 & 0 & \hdots & 0 & 1 \\
1 & 0 & 0 & \hdots & k & 0
\end{array} \right) \]
Up to a sign, the characteristic polynomial of $B$ is
\[ x^N - k x^{N-2} - 1 = q_k(x). \]
Again, we chose $k$ so that this polynomial is irreducible and so $\spectralratio(\phi_k) = \spectralratio(q_k)$.
Thus by Lemma~\ref{lem:root_pos} this polynomial has $N-2$ roots that lie in the interior of the unit circle, one root in $[-\sqrt{k+1}, -\sqrt{k-1}]$ and one root in $[\sqrt{k-1}, \sqrt{k+1}]$.
Thus, $\stretchfactor(\phi_k) \approx \sqrt{k}$ and so
\[ \spectralratio(\phi_k^{-1}) = \spectralratio(q_k) \leq \frac{\sqrt{k + 1}}{\sqrt{k - 1}} \leq 1 + \frac{1}{\sqrt{k}} \]
as required.

It also follows from this computation that $\stretchfactor(\phi_k) \neq \stretchfactor(\phi_k^{-1})$ and so $\phi_k$ is a non-geometric automorphism.

Now suppose that $f$ has no periodic Nielsen paths \cite[Page~25]{Pfaff}.
Note that the local Whitehead graph of $f$ is connected \cite[Page~25]{Pfaff} and $A$ is a Perron--Frobenious matrix.
Thus by the Full Irreducibility Criterion \cite[Lemma 9.9]{Pfaff} we have that $\phi_k$ is fully irreducible.

On the other hand, if $f$ has a periodic Nielsen path then $\phi_k$ must be parageometric \cite[Page~3155]{HandelMosher}.
Thus $\phi_k^{-1}$ is not parageometric \cite[Corollary~2]{HandelMosherParageometric} and so $g$ has no periodic Nielsen paths.
Again, the local Whitehead graph of $g$ is connected and $B$ is a Perron--Frobenious matrix, hence $\phi_k^{-1}$ is fully irreducible.

In either case, $\phi_k$ and $\phi_k^{-1}$ are both fully irreducible.
\end{proof}

Handel and Mosher described how there is no uniform upper bound, depending only on $N$, on the distance between the axes of $\phi$ and $\phi^{-1}$ \cite[Section~7.1]{HandelMosherAxes}.
We finish by noting that $\phi_k$ is an explicit family of such outer automorphisms.

\begin{theorem}
Suppose that $k \geq 5$ is such that $p_k$ and $q_k$ are both irreducible.
There are axes for $\phi_k$ and $\phi_k^{-1}$ which are separated by at least $\left(N-\frac{5}{2} \right) \log(k) - 6$.
\end{theorem}

\begin{proof}
For ease of notation let $\lambda_k \defeq \stretchfactor(\phi_k)$ and $\lambdabar_k \defeq \stretchfactor(\phi_k^{-1})$.

Begin by noting that the Perron--Frobenious eigenvectors of the matrices $A$ and $B$ in the proof of Theorem~\ref{thrm:spectral_ratio_higher} are
\[ \left( \begin{array}{c}
\lambda_k^{N-1} \\
\lambda_k^{N-2} \\

\vdots \\
\lambda_k \\
1
\end{array} \right)
\inlineand
\left( \begin{array}{c}
1 \\
\lambdabar_k \\
\vdots \\
\lambdabar_k^{N-2} \\
\lambdabar_k^{N-1}
\end{array} \right)
\]
respectively.
Let $X_0$ and $Y_0$ be the metric graphs in Outer space corresponding to $N$--petalled roses with these lengths assigned.
Then $X_0$ and its images $X_i \defeq X_0 \cdot \phi_k^i$ lie on an axis of $\phi_k$.
Similarly, $Y_0$ and its images $Y_i \defeq Y_0 \cdot \phi_k^{-i}$ lie on an axis of $\phi_k^{-1}$.

We note that
\[ d(X_i, X_{i+1}) = \log(\lambda_k) \inlineand d(Y_i, Y_{i+1}) = \log(\lambdabar_k). \]
Furthermore, direct calculation of the lengths of candidate loops shows that $d(X_i, Y_i) \leq d(X_i, Y_j)$ for every $i$ and $j$ and that
\[ d(X_i, Y_i) = \log \left(\frac{\lambda_k^N - 1}{\lambda_k - 1} \cdot \frac{\lambdabar_k^N - \lambdabar_k^{N-1}}{\lambdabar_k^N - 1} \right) \geq (N-1) \log(k) - 4. \]

Therefore suppose that the distance between these two axes is realised by points $X$ and $Y$.
Without loss of generality we may assume that $X$ lies in the segment $[X_i, X_{i+1}]$ and that $Y$ lies in the segment $[Y_{j-1}, Y_j]$.
However, as shown in Figure~\ref{fig:axes},
\[ d(X, Y) \geq d(X_i, Y_j) - d(X_i, X) - d(Y, Y_j). \]
As $\lambda_k \leq k+1$ and $\lambdabar_k \leq \sqrt{k + 1}$, the right hand side of this inequality is bounded below by
\[ (N-1) \log(k) - 4 - \frac{3}{2} \log(k + 1) \geq \left(N-\frac{5}{2} \right) \log(k) - 6 \]
as required.
\end{proof}

\begin{figure}
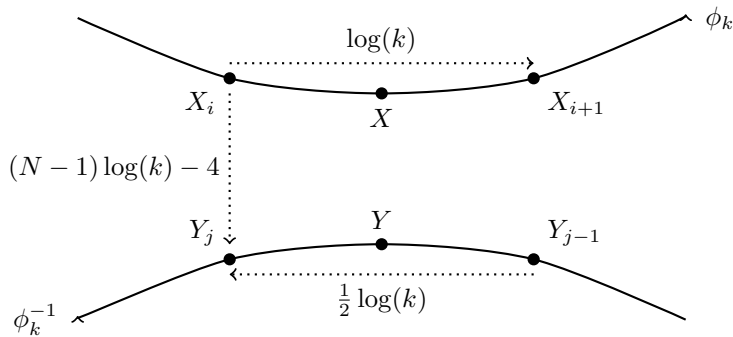

\centering
\include{tikz_axes}
\caption{Axes of $\phi_k$ and $\phi_k^{-1}$.}
\label{fig:axes}
\end{figure}

We note that for each $\phi_k$ the turns $\{a_i, a_j\}$ are all illegal.
Hence these are never lone axis automorphisms \cite[Theorem~3.9]{MosherPfaff}.
Similarly, for each $\phi_k^{-1}$ the turns $\{a_{i}^{-1}, a_{i+2}\}$ (where the indices wrap over so $a_{N+1} \defeq a_1$ and $a_{N+2} \defeq a_2$) are all illegal and so these are also never lone axis automorphisms.

\begin{question}
Are there axes of $\phi_k$ and $\phi_k^{-1}$ that remain within a uniformly bounded distance of each other?
\end{question}

\begin{remark}
From the eigenvectors above, we also see that as $k$ goes to infinity these axes enter thinner and thinner parts of Outer space.
\end{remark}

\begin{acknowledgements}
The author wishes to thank Yael Algom-Kfir and Christopher J. Leininger for many helpful discussions regarding this result.
\end{acknowledgements}

\bibliographystyle{plain}
\bibliography{bibliography}

\end{document}

%% file: settings.tex
\usepackage{url}

\usepackage{colonequals}
\usepackage{pgfplots}

\usepackage{amsmath, amsthm, amssymb}
\theoremstyle{plain}
\newtheorem{theorem}[equation]{Theorem}

\newtheorem{lemma}[equation]{Lemma}

\theoremstyle{definition}
\newtheorem{definition}[equation]{Definition}
\newtheorem{remark}[equation]{Remark}
\newtheorem{question}[equation]{Question}

\newtheorem*{keywords}{keywords}
\newtheorem*{acknowledgements}{Acknowledgements}

\newtheoremstyle{dotless}{}{}{}{}{\bfseries}{}{ }{}
\theoremstyle{dotless}
\newtheorem*{ccode}{Mathematics Subject Classification 2010:}

\newcommand{\fakeenv}{} 

{ 
 \renewcommand{\fakeenv}{#2} 
 \theoremstyle{plain} 
 \newtheorem*{\fakeenv}{#1~\ref{#2}} 
 \begin{\fakeenv}
}
{
 \end{\fakeenv}
}

\DeclareMathOperator{\Mod}{Mod^+}
\DeclareMathOperator{\Out}{Out}

\newcommand{\lambdabar}{{\mkern0.75mu\mathchar '26\mkern -9.75mu\lambda}}

\newcommand{\calC}{\mathcal{C}}

\newcommand{\outerspace}{{CV}}
\newcommand{\FF}{\mathbb{F}}

\newcommand{\QQ}{\mathbb{Q}}

\newcommand{\ZZ}{\mathbb{Z}}
\newcommand{\RP}{\mathbb{RP}}

\newcommand{\defeq}{\colonequals}

\newcommand{\inlineand}{\quad \textrm{and} \quad}

\DeclareMathOperator{\stretchfactor}{\lambda}
\DeclareMathOperator{\spectralratio}{\omega}

\usepackage{tikz}
\usetikzlibrary{calc}
\usetikzlibrary{decorations.pathreplacing}
\usetikzlibrary{shapes.geometric}

\usepackage[pdfborderstyle={},pdfborder={0 0 0},pagebackref,pdftex]{hyperref}

\renewcommand*{\backref}[1]{}
\renewcommand*{\backrefalt}[4]{
  \ifcase #1 %
   [No citations.]%
  \or
   [#2]%
  \else
   [#2]%
  \fi
}

%% file: tikz_axes.tex
\begin{tikzpicture}[scale=2,thick]

\tikzset{dot/.style={draw,shape=circle,fill=black,scale=0.4}}

\node (Xs) at (-2, 1) {};
\node [dot] (X) at (-1, 0.6) [label=below left:{$X_i$}] {};
\node [dot] (X0) at (0, 0.5) [label=below:{$X$}] {};
\node [dot] (Xp) at (1, 0.6) [label=below right:{$X_{i+1}$}] {};
\node (Xe) at (2, 1) [label=right:{$\phi_k$}] {};

\node (Ys) at (-2, -1) [label=left:{$\phi_k^{-1}$}] {};
\node [dot] (Y) at (-1, -0.6) [label=above left:{$Y_j$}] {};
\node [dot] (Y0) at (0, -0.5) [label=above:{$Y$}] {};
\node [dot] (Yp) at (1, -0.6) [label=above right:{$Y_{j-1}$}] {};
\node (Ye) at (2, -1) {};

\draw [->] plot [smooth] coordinates {(Xs) (X) (X0) (Xp) (Xe)};
\draw [->] plot [smooth] coordinates {(Ye) (Yp) (Y0) (Y) (Ys)};

\draw [dotted, ->] ([yshift=1mm]X.center) -- ([yshift=1mm]Xp.center) node [midway,above] {$\log(k)$};
\draw [dotted, ->] ([yshift=-1mm]Yp.center) -- ([yshift=-1mm]Y.center) node [midway,below] {$\frac{1}{2}\log(k)$};
\draw [dotted, ->] ([yshift=-1mm]X.center) -- ([yshift=1mm]Y.center) node [midway,left] {$(N - 1)\log(k) - 4$};

\end{tikzpicture}